\documentclass{amsart}
\usepackage[l2tabu,orthodox]{nag}
\usepackage{geometry}
\usepackage{amssymb,latexsym,enumerate,verbatim}

\usepackage{tikz}
\usepackage{tikz-cd}
\usepackage{enumitem}
\usepackage{amsmath}
\usepackage{graphicx}

\usepackage[
backend=biber,
style=alphabetic,
sorting=nty
]{biblatex}
\usepackage{hyperref}
\hypersetup{
    colorlinks=true,
    linkcolor=blue,
    }

\theoremstyle{plain}
\newtheorem{Thm}[equation]{Theorem}
\newtheorem{Cor}[equation]{Corollary}
\newtheorem{Lem}[equation]{Lemma}
\newtheorem{Prop}[equation]{Proposition}

\theoremstyle{definition}
\newtheorem{Def}[equation]{Definition}

\newtheorem{Rem}[equation]{Remark}

\numberwithin{equation}{section}

\DeclareMathOperator{\im}{im}

\newcommand{\F}{\mathbb{F}}

\DeclareMathOperator{\Conf}{Conf}

\DeclareMathOperator{\Sym}{Sym}

\DeclareMathOperator{\tr}{tr}

\DeclareMathOperator{\id}{id}

\title{On the wreath transfer in the homology of symmetric groups}
\author{Dezhou Li}
\date{July 2026}

\begin{document}

\begin{abstract}
For an odd prime $p$, we give an explicit formula for the wreath transfer
\[
\tr\colon H_*(\Sigma_{p^2};\F_p)\longrightarrow
H_*(\Sigma_p\wr\Sigma_p;\F_p)
\]
associated with the standard inclusion $\Sigma_p\wr\Sigma_p\leq \Sigma_{p^2}$. The proof is a direct algebraic computation, using Cartan's classical double-coset formula together with a generating-function argument. As an immediate consequence, we obtain an elementary proof of the validity of the mixed Adem relations in Goodwillie calculus.
\end{abstract}

\maketitle

\section{Introduction}
The inclusion of the wreath product $\Sigma_p \wr \Sigma_p$ into $\Sigma_{p^2}$ plays a role in both classical topology, where the kernel of the induced map on homology encodes the Adem relations \cite{CLM1976}, and in recent studies such as calculations of Goodwillie towers \cite{AM1999} and Lie power operations \cite{Kjaer}. 
Moreover, the transfer associated with this inclusion is the canonical wrong-way map
$$
\tr:
H_*(\Sigma_{p^2};\mathbb F_p)\longrightarrow
H_*(\Sigma_p\wr\Sigma_p;\mathbb F_p),
$$
which is the primary object of study in this paper.

Our geometric motivation comes from the study of configuration spaces (for comprehensive background, see \cite{Knudsen2018}). Consider the $\Sigma_p \wr \Sigma_p$-equivariant projection map $\Conf_{p^2} \rightarrow (\Conf_p)^p$ over Euclidean space $\mathbb{R}^n$. This geometric map induces a homomorphism between their associated spectral sequences. On the $E_2$ page, this induced map takes the form:
$$H_s(\Sigma_{p^2}; H_t(\Conf_{p^2})) \longrightarrow H_s(\Sigma_p \wr \Sigma_p; H_t((\Conf_p)^p)).$$
At the bottom row where $t=0$, this yields the map$$H_s(\Sigma_{p^2}; \F_p) \longrightarrow H_s(\Sigma_p \wr \Sigma_p; \F_p),$$
which coincides precisely with the transfer homomorphism. Understanding this transfer map is a prerequisite for providing an alternative geometric proof of the Arone--Mahowald theorem and for explicitly computing the homology of the orbit space $B_k := \Conf_k /\Sigma_k$.

To evaluate this map, we must describe the basis elements for the $\F_p$-vector spaces on both sides. 
It is a classical result that $H_i(C_p) \cong \mathbb{F}_p$ for $i \geq 0$; let $e_i$ denote the canonical generator in degree $i$. By a slight abuse of notation, we will also use $e_i$ interchangeably to denote the corresponding generator of $H_i(\Sigma_p)$. Over $\F_p$, the basis elements of $H_*(\Sigma_{p^2})$ (see Theorem \ref{DLThm} for details) are represented by strongly admissible words in the Dyer--Lashof operations, taking the form:
$$\beta^{\epsilon_1} Q^{s_1} \circ \beta^{\epsilon_2} Q^{s_2}.$$
Strong admissibility requires that $s_i \in \mathbb{Z}$ for $i \in \{1,2\}$, and that $ps_2 - \epsilon_2 \geq s_1 > (p-1)s_2$.
On the other hand, the target space $H_*(\Sigma_p \wr \Sigma_p)$ contains pure tensor powers of the form $e_i \otimes e_j^{\otimes p}$ for $i, j \geq 0$. With the normalization fixed in Remark \ref{ImportantRem}, we write $Q^i\wr Q^j$ for the corresponding pure wreath tensor. The composition $Q^r \circ Q^s \in H_*(\Sigma_{p^2})$ is defined as the image of $Q^r \wr Q^s$ under the map $i_*: H_*(\Sigma_p \wr \Sigma_p) \rightarrow H_*(\Sigma_{p^2})$ induced by inclusion.

A natural problem arises regarding the explicit evaluation of this transfer map. Specifically, one wishes to determine the exact coefficients $c_{i,j} \in \F_p$ in the expansion:
$$\tr(Q^r \circ Q^s) = \sum c_{i,j} Q^i \wr Q^j.$$

For the prime $p=2$, the explicit formula was first provided by Priddy \cite{Priddy1973} without a formal proof. Subsequently, Kuhn computed this specific transfer in \cite[Example 7.6]{Kuhn} as a concrete demonstration of his structural double coset decomposition techniques. Furthermore, Behrens \cite[Lemma 1.4.3]{Behrens} simplified this formula to the explicit algebraic expansion:
$$\tr(Q^r \circ Q^s) = Q^r \wr Q^s + \sum_{l=0}^{r-s-1} \binom{2s-r+1+2l}{l} Q^{2s+1+l} \wr Q^{r-s-1-l}.$$

For odd primes $p > 2$, obtaining an explicit formula for this transfer has remained a longstanding open problem of topological interest. Establishing this general formula yields several immediate applications.  In addition to its relevance to configuration spaces, explicitly determining the coefficients of this transfer expansion is equivalent to establishing the ``mixed Adem relations" originally conjectured by Kjaer \cite[Conjecture 4.3]{Kjaer}. Konovalov \cite[Theorem 8.2.15]{Kono2025} has recently given a proof of Kjaer's conjecture using the machinery of functor calculus and the algebraic Goodwillie spectral sequence. 

Alongside deriving the exact formula for the wreath transfer, the aim of the present paper is to provide an alternative proof of Kjaer's conjecture through a purely algebraic and explicit computation of the transfer map. In this sense, our approach provides a direct and elementary proof of the mixed Adem relations.

With this context established, we state the main result of the paper.

\begin{Thm}[Wreath transfer formula]\label{mainThm}
    For the inclusion $\Sigma_{p} \wr \Sigma_{p} \le \Sigma_{p^{2}}$, the transfer homomorphism
    $$\tr: H_{*}(\Sigma_{p^{2}}; \mathbb{F}_{p}) \longrightarrow H_{*}(\Sigma_{p} \wr \Sigma_{p}; \mathbb{F}_{p})$$
    sends
    $$\beta^{\epsilon}Q^{r} \circ Q^{s} \mapsto \beta^{\epsilon}Q^{r} \wr Q^{s} + \sum_{0 \leq j < s} (-1)^{r-j}\binom{(j-s)(p-1)-1}{r-(p-1)s-j-1} \beta^{\epsilon}Q^{r+s-j} \wr Q^{j},$$
    where $p > 2$, $s(p-1) < r \leq sp$ and $\epsilon \in \{0,1 \}$. The binomial coefficient in the displayed formula is interpreted in the
    generalized sense, i.e.,
    $$
    \binom{-N}{k}
    =
    (-1)^k
    \binom{N+k-1}{k} \quad\text{if }N > 0, k \geq 0; \qquad   \binom{N}{k}=0
    \quad\text{if }k<0.
    $$
\end{Thm}

\begin{Rem}
    By Lemma \ref{indexProp}, the composition $i_* \circ \tr$ acts as multiplication by the index $[\Sigma_{p^2} : \Sigma_p \wr \Sigma_p]$. Because this index is congruent to $1 \pmod p$, applying $i_*$ to the expansion above yields $\beta^{\epsilon} Q^r \circ Q^s$. Consequently, the image of the residual summation under $i_*$ must vanish; this vanishing is explained by the classical Adem relations.
\end{Rem}

We now state Kjaer's conjecture using this transfer formula for odd $p$. Adapting Kjaer's notation, for a spectrum $X$, define
$$
\mathbb D_n(X):=(\partial_n\wedge X^{\wedge n})_{h\Sigma_n},
$$
where $\partial_n$ denotes the $n$-th spectrum of the spectral Lie operad and $(-)_{h\Sigma_n}$ denotes homotopy orbits.
In the corollary below, $\mathbb{S}^{2l+1} = \Sigma^{2l+1} \mathbb{S}^0$ is a sphere spectrum.

\begin{Cor}\label{KjaerCor}
Let $l\geq 0$, and let $\iota\in H_{2l+1}(\mathbb{S}^{2l+1};\mathbb F_p)$ denote its generator. In
$H_*(\mathbb D_{p^2}(\mathbb{S}^{2l+1});\mathbb F_p),$
we have
$$
\beta^\epsilon Q^rQ^s\iota
=
-\sum_{0\leq j<s}
(-1)^{r-j}
\binom{(j-s)(p-1)-1}{r-(p-1)s-j-1}
\beta^\epsilon Q^{r+s-j}Q^j\iota,
$$
where $p > 2$, $s(p-1)<r\leq sp$ and $\epsilon\in\{0,1\}$.
\end{Cor}

\begin{Rem}
    This agrees with Konovalov's relation \cite[Definition 8.2.4]{Kono2025} after the change of variables. As he noted in \cite[Remark 8.2.17]{Kono2025}, the formula for the transfer was unknown and could be used to give an alternative proof of Kjaer's conjecture.
\end{Rem}

To prove Theorem \ref{mainThm}, we use Cartan's double-coset formula. Applied to the subgroup tower
\[
C_p\times C_p \leq C_p\wr C_p \leq \Sigma_p\wr\Sigma_p \leq \Sigma_{p^2},
\]
this formula decomposes the global transfer into a sum of local terms induced by conjugation. In this case, the only nonzero terms are represented by the set
$S=\{e,\gamma,\tau^\lambda\mid \lambda\in\F_p^\times\}$,
consisting of the identity element, the swap element, and the $p-1$ shear elements. Thus the computation is organized by the following diagram:
\[
\begin{tikzcd}
	{H_*(C_p \times C_p)} & {H_*(C_p \wr C_p)} & {H_*(\Sigma_p \wr \Sigma_p)} & {H_*(\Sigma_{p^2})} \\
	\\
	{H_*(C_p \times C_p)} && {H_*(C_p \times C_p)} & {H_*(\Sigma_p \wr \Sigma_p)}
	\arrow["{(\iota_{\Delta})_*}", from=1-1, to=1-2]
	\arrow["\id"', from=1-1, to=3-1]
	\arrow["{i_*}", from=1-2, to=1-3]
	\arrow["{i_*}", from=1-3, to=1-4]
	\arrow["\tr", from=1-4, to=3-4]
	\arrow["{\sum_{g \in S}(c_{g^{-1}})_*}"', from=3-1, to=3-3]
	\arrow["{i_* \circ (\iota_{\Delta})_*}"', from=3-3, to=3-4]
\end{tikzcd}
\]
Thus it remains to compute these local contributions separately and then sum the resulting coefficients to obtain the global transfer formula.

\subsection{Conventions}
Throughout this paper, all homology groups are computed with coefficients in the finite field $\F_p$ for an odd prime $p > 2$. For notational brevity, we adopt the following abbreviations: $G = \Sigma_{p^2}$, $H = \Sigma_p \wr \Sigma_p$, $K = C_p \wr C_p$, and $L = C_p \times C_p$.

\subsection{Organization of the paper}
In Section 2, we review some group homology tools used in the proof. We recall the inclusion and transfer maps, Cartan's double-coset formula, and the vanishing criterion for transfers between elementary abelian $p$-groups. We then introduce the subgroup tower inside $\Sigma_{p^2}$, and identify the double-coset representatives that contribute to the transfer calculation. The section also reviews the homology of wreath products, May's formula for the standard diagonal embedding, the Dyer--Lashof indexing conventions, and the conjugation formulas associated with the identity, shear, and swap representatives.

In Section 3, we prove the wreath transfer formula. We first choose a preimage element $X(r,s;\epsilon)\in H_*(C_p\times C_p)$ whose image under the standard diagonal embedding represents the class to which the transfer is applied. We then use Cartan's double-coset formula to express the transfer coefficient as a sum of three terms: the identity contribution $I_j$, the shear contribution $S_j$, and the swap contribution $G_j$. 
The section evaluates the boundary range $j\geq s$, reduces the interior range $0\leq j<s$ to a separate coefficient identity, and treats the Bockstein case. Together with the coefficient identity proved in Section 4, these results prove Theorem \ref{mainThm} and Corollary \ref{KjaerCor}.

In Section 4, we prove the coefficient identity needed for the interior range $0\leq j<s$. The proof is a generating-function calculation. The shear contribution and the swap contribution are first rewritten as coefficient extractions involving the same auxiliary $P$-section polynomial $D_A(z)$. An averaging identity over $\mathbb F_p^\times$ and a coefficient-extraction lemma are then used to compare the two expressions. Finally, the common $P$-section $D_A(z)$ terms cancel, leaving the required binomial coefficient and completing the proof of the interior coefficient formula.

\subsection{Acknowledgments}
I am grateful to my Ph.D. advisor, Ben Knudsen, for his guidance and support. The idea for this paper grew out of questions arising in the study of configuration spaces. I would also like to thank Guchuan Li for his support and for helping me extend the formula to the general case. Finally, I thank Dev Sinha for helpful conversations and encouragement. I used ChatGPT for language editing and for assistance with the coefficient-extraction formulation of the swap contribution (Proposition \ref{swapcontri}) in Section 4. All mathematical content was checked by the author.

\section{Preliminaries on group homology}
In this section, we review essential tools from group homology.
For a subgroup $H \leqslant G$, we have the natural map induced by inclusion and the wrong-way transfer map:
\begin{align*} 
i_{H \rightarrow G} &: H_*(H) \rightarrow H_*(G) \\ 
\tr_{G \rightarrow H} &: H_*(G) \rightarrow H_*(H)
\end{align*}
While there are many notational conventions for these maps, we adopt the notation from \cite{KP1977}. Furthermore, when the context is clear, we will simply write $i_*$ for the induced map and $\tr$ for the transfer map.

\subsection{Cartan's double-coset formula}
We shall recall some facts about the inclusion and transfer.
The most important one is Cartan's double-coset formula, which allows us to decompose global transfers into local conjugations. 

We first recall a vanishing criterion for the transfer map in the context of elementary abelian groups.

\begin{Lem}[{\cite[Lem 3.1]{KP1977}}]\label{0Lem}
    Let $G$ be an elementary abelian $p$-group and let $H$ be a proper subgroup of $G$. Then $\tr_{G \rightarrow H} = 0$.
\end{Lem}

We frequently need to compose the transfer map with the induced map.

\begin{Prop}[{\cite[Chapter 3, Prop 9.5]{Brown}}]\label{indexProp}
    Suppose $H \leqslant G$ and $[G : H] < \infty$. Then
    $$i_{H \rightarrow G} \circ \tr_{G \rightarrow H} = [G : H] \id.$$
\end{Prop}

The primary computational tool for this paper is the double-coset formula.

\begin{Prop}[{\cite[Prop 3.2]{KP1977}}]\label{doublecosetProp}
    Suppose $H$ and $L$ are two subgroups of $G$ with $[G : H] < \infty$ and $[G : L] < \infty$. Then,
    $$\tr_{G \rightarrow H} \circ i_{L \rightarrow G} = \sum_{g \in E} i_{H \cap g^{-1} L g \rightarrow H} \circ (c_{g^{-1}})_* \circ \tr_{L \rightarrow g H g^{-1} \cap L}$$
    where $E$ is a set of representatives for the double cosets $L\backslash G/H$ and 
    $$c_{g^{-1}}: gHg^{-1} \cap L \rightarrow H \cap g^{-1}Lg, \quad x \mapsto g^{-1}xg. $$
\end{Prop}

\subsection{A subgroup tower for the symmetric group}\label{subgrouptowersubsection}
In this subsection, we define a subgroup tower of $G = \Sigma_{p^2}$ by considering its natural action on a $p \times p$ grid of points, and simplify the double-coset sum.

Let $\Omega = \F_p^2$ be the $p \times p$ grid of points $(x,y)$. Define $G = \Sym (\Omega) \cong \Sigma_{p^2}$.
Thus $G$ is the group of all permutations of the $p^2$ points of $\Omega$. For each $x \in \F_p$, define the vertical block $$B_x := \{x \} \times \F_p.$$
Then, $\{ B_x \mid x \in \F_p \}$ is a partition of $\Omega$ into $p$ blocks, each containing $p$ points. We let the group $H$ be the setwise stabilizer of this partition. Thus, $H = \Sigma_p \ltimes (\Sigma_p)^p = \Sigma_p \wr \Sigma_p$ and we may define its Sylow $p$-subgroup $K = C_p \wr C_p$. Moreover, we define $\alpha(x,y) = (x+1, y)$ and $\beta(x,y) = (x, y+1)$. We let the group $L = \langle \alpha, \beta \rangle \cong C_p \times C_p$, and the action of $L$ on $\Omega$ is regular.

The translation $\alpha$ sends $B_{x} \mapsto B_{x+1}$, so it cyclically permutes the $p$ blocks. The translation $\beta$ preserves each block and acts on every block by the same $p$-cycle sending $y \mapsto y+1$. Hence, $L \hookrightarrow K$. In wreath-product coordinates, the inclusion is the standard diagonal embedding:
$$\iota_{\Delta}: C_p \times C_p \rightarrow C_p \ltimes (C_p)^p,$$ such that
$$\iota_{\Delta}(a,b) = (\sigma_{a}; \rho_{b}, \rho_{b}, \dots, \rho_{b}),$$
where $\sigma_{a}(x) = x + a$ and $\rho_{b}(y) = y + b$. Hence, we have the following tower of subgroups:
$$L = C_p \times C_p \hookrightarrow K = C_p \wr C_p \hookrightarrow H = \Sigma_p \wr \Sigma_p \hookrightarrow G = \Sigma_{p^2}.$$

We may simplify the double-coset sum using the following result.

\begin{Lem}[{\cite[Prop 2.5]{KP1977}}]\label{conjLem}
    Let $E \leq H = \Sigma_p \wr \Sigma_p$ be an elementary abelian subgroup of order $p^2$ acting transitively on $\Omega$. Then $E$ is conjugate inside $H$ to the subgroup $L = C_p \times C_p$.
\end{Lem}

We can apply Cartan's double coset formula in our setting.
By Lemma \ref{0Lem}, since $L$ is an elementary abelian $p$-group, we have $$\tr_{L \rightarrow g H g^{-1} \cap L} = 0$$ if $g H g^{-1} \cap L \neq L$, and $$\tr_{L \rightarrow g H g^{-1} \cap L} = \id$$ if $g H g^{-1} \cap L = L$. Moreover, $g H g^{-1} \cap L = L$ is equivalent to $g^{-1} L g \leq H$. If we let $E := g^{-1} L g$, then by Lemma \ref{conjLem}, there is $h \in H$ such that $h E h^{-1} = L$. Set $n := gh^{-1}$, then $$n^{-1} L n = h E h^{-1} = L.$$ Hence, $n \in N_G(L)$. Also, $n = g h^{-1} \in g H$, so $LnH=LgH$. Therefore the double coset $LgH$ has a representative in $N_G(L)$. Replacing $g$ by this representative, we may assume from now on that $g \in N_G(L)$ and
the local term in Cartan's formula reduces to the composition of 
$$
H_*(L)\xrightarrow{(c_{g^{-1}})_*}H_*(L)\xrightarrow{i_{L\rightarrow H}}H_*(H),$$
where $i_{L \rightarrow H}$ factors through $K$:
$$H_*(L) \xrightarrow{(\iota_{\Delta})_*} H_*(K) \xrightarrow{i_*} H_*(H).$$

\subsection{The relevant double-coset representatives}
To describe the double coset representative $g$ such that $g \in N_G(L)$, we use the following classification due to Kuhn \cite{Kuhn}.

\begin{Prop}\label{KProp}
    Let $G = \Sigma_{p^2}$, $H = \Sigma_p \wr \Sigma_p$, and $L = C_p \times C_p$. There is a one-to-one correspondence between the double cosets in $L\backslash G/H$ having a representative in $N_G(L)$ and the cosets in $GL(2,\F_p)/\mathbb{L}$:
    $$
    \{LgH \mid g \in N_G(L)\}\subseteq L\backslash G/H
    \longleftrightarrow GL(2,\F_p)/\mathbb{L}.
    $$
where $\mathbb{L}$ is the subgroup of invertible lower-triangular matrices in $GL(2,\F_p)$.
\end{Prop}

\begin{Rem}\label{KLem}
    Note that $| GL(2, \F_p)| = (p^2-1)(p^2-p)$ and $|\mathbb{L}| = (p-1)^2p$, so $|GL(2, \F_p)/\mathbb{L}| = p+1$, which means that there are $p+1$ double coset representatives in $L \backslash G / H$ such that they are in the normalizer of $L$ in $G$.
\end{Rem}
For $\lambda = 1, 2, \dots, p-1$, the matrix representations in $GL(2, \F_p)/ \mathbb{L}$ are $$\begin{bmatrix}
1 & 0 \\
0 & 1 
\end{bmatrix}, \begin{bmatrix}
0 & 1 \\
1 & 0 
\end{bmatrix},\begin{bmatrix}
1 & \lambda \\
0 & 1 
\end{bmatrix}.$$  
They correspond to the identity element $e$, the ``swap'' element $\gamma$ and the ``shear'' elements $\tau^{\lambda}$, where $\gamma(x,y) = (y,x)$ and $\tau^{\lambda}(x,y) = (x + \lambda y, y)$. 
Specifically, the matrix 
$\begin{bmatrix} 0 & 1 \\ 1 & 0 \end{bmatrix}$ 
acts on basis vectors by:
$$\begin{bmatrix} 1 \\ 0 \end{bmatrix} \mapsto \begin{bmatrix} 0 \\ 1 \end{bmatrix} \quad \text{and} \quad \begin{bmatrix} 0 \\ 1 \end{bmatrix} \mapsto \begin{bmatrix} 1 \\ 0 \end{bmatrix}.$$
In $G$, we observe the conjugation relations $\gamma \alpha \gamma^{-1} = \beta$ and $\gamma \beta \gamma^{-1} = \alpha$. Thus, under Kuhn's correspondence, this matrix maps uniquely to the swap element $\gamma \in G$.
Similarly, the matrix 
$\begin{bmatrix} 1 & \lambda \\ 0 & 1 \end{bmatrix}$ 
acts by:
$$\begin{bmatrix} 1 \\ 0 \end{bmatrix} \mapsto \begin{bmatrix} 1 \\ 0 \end{bmatrix} \quad \text{and} \quad \begin{bmatrix} 0 \\ 1 \end{bmatrix} \mapsto \begin{bmatrix} \lambda \\ 1 \end{bmatrix}.$$
This matrix corresponds to the shear element $\tau^{\lambda}$, which satisfies the conjugation relations $\tau^{\lambda} \alpha (\tau^{\lambda})^{-1} \\ =\alpha$ 
and 
$\tau^{\lambda} \beta (\tau^{\lambda})^{-1} = \alpha^{\lambda} \beta$.

Combining the arguments following Lemma \ref{conjLem} with Proposition \ref{KProp}, we can reduce the original double-coset sum to the following form:

\begin{Cor}\label{doublecosetCor}
    Let $X \in H_*(C_p \times C_p)$, set $x=(\iota_{\Delta})_*(X) \in H_*(C_p \wr C_p)$, and set $z=i_*(x) \in H_*(\Sigma_{p^2})$. Then
    $$\tr_{\Sigma_{p^2} \rightarrow \Sigma_p \wr \Sigma_p}(z)
    = \sum_{g \in S} i_* \circ (\iota_{\Delta})_* \circ (c_{g^{-1}})_*(X),$$
    where $S = \{e, \gamma, \tau^{\lambda} \mid \lambda \in \F_p^{\times} \}$. Equivalently, the following diagram is commutative:
    \[\begin{tikzcd}
	{H_*(C_p \times C_p)} & {H_*(C_p \wr C_p)} & {H_*(\Sigma_p \wr \Sigma_p)} & {H_*(\Sigma_{p^2})} \\
	\\
	{H_*(C_p \times C_p)} && {H_*(C_p \times C_p)} & {H_*(\Sigma_p \wr \Sigma_p)}
	\arrow["{(\iota_{\Delta})_*}", from=1-1, to=1-2]
	\arrow["\id"', from=1-1, to=3-1]
	\arrow["{i_*}", from=1-2, to=1-3]
	\arrow["{i_*}", from=1-3, to=1-4]
	\arrow["\tr", from=1-4, to=3-4]
	\arrow["{\sum_{g \in S}(c_{g^{-1}})_*}"', from=3-1, to=3-3]
	\arrow["{i_* \circ (\iota_{\Delta})_*}"', from=3-3, to=3-4]
\end{tikzcd}\]
\end{Cor}

\subsection{Homology of wreath products and the standard diagonal embedding}
In this section, we provide a formula for the standard diagonal embedding map $(\iota_{\Delta})_*$ from the previous section. To do this, we must first describe the basis elements in the homology of the wreath product.

First, it is a well-known fact that over $\F_p$, $H_i(C_p ; \F_p) \cong \F_p$ with generator $e_i$ for $i \geq 0$. Next, if $H_*(H)$ has the basis $x_0 = 1, x_1, x_2, \dots$ over $\F_p$, then $H_*(C_p \wr H)$ has the basis over $\F_p$ (\cite[Prop 2.2]{DL}):
\begin{align}\label{wrbasis}
    & e_i \otimes x_j \otimes x_j \otimes \dots \otimes x_j = e_i \otimes (x_j)^{\otimes p}  \ (i, j \geq 0),\\
    & e_0 \otimes x_{i_1} \otimes x_{i_2} \otimes \dots \otimes x_{i_p} \ (i_k \neq i_l \ \text{for some} \ k, l),
\end{align}
where $(i_1, \dots,i_p)$ runs through each representative of the classes obtained by cyclic permutations of the indices. We call the elements in the first family pure wreath tensors and those in the second family mixed wreath tensors.

We define the coefficients appearing in May's formula as follows. For integers $j$ and $k$, set
\begin{equation}\label{tEqu}
    t(j,k) = (-1)^k v(j) \binom{\lfloor j/2 \rfloor - k(p-1)}{k},
\end{equation}
where
\[
\delta(j) =  \begin{cases} 
      1 & j \ \text{odd}, \\
      0 & j \ \text{even}.
   \end{cases}
\]
and
\begin{equation}\label{vEqu}
    v(j) = (-1)^{\frac{j(j-1)(p-1)}{4}}(\frac{p-1}{2}!)^j
\end{equation}
with $v(2j + \epsilon) \equiv (-1)^j (\frac{p-1}{2}!)^{\epsilon} \pmod p$ and $\epsilon \in \{0, 1\}$. Note that Wilson's theorem implies $(\frac{p-1}{2}!)^2 \equiv (-1)^{\frac{p-1}{2}+1} \pmod p$.

The following formula for the standard diagonal embedding map $(\iota_{\Delta})_*$ is due to Peter May.

\begin{Lem}[{\cite[Lem 4.6]{May}}]\label{Mayformula}
    $H_*(C_p \wr C_p) \cong H_*(C_p; H_*(C_p)^{\otimes p})$ and the formula for $(\iota_{\Delta})_*: H_*(C_p \times C_p) \rightarrow H_*(C_p \wr C_p)$ is given by
    \begin{align*}
    e_x \otimes e_y \mapsto & \sum_{k \geq 0} t(y,k) \ e_{x + (2pk-y)(p-1)} \otimes e^{\otimes p}_{y - 2k(p-1)} \\
    &- \delta(x) \delta(y-1)
    \sum_{k \geq 0}t(y-1,k) \ e_{x + p + (2pk-y)(p-1)} \otimes e^{\otimes p}_{y-2k(p-1)-1},
    \end{align*}
    where we declare $e_q  = 0$ for $q < 0$.
\end{Lem}

\subsection{Lower and upper Dyer--Lashof indices}
Dyer--Lashof operations provide a framework for understanding the homology of symmetric groups. We begin by recalling the standard definitions of the lower and upper indexing conventions.

\begin{Def}\label{lowerindexDLDef}
    For $i \geq 0$,
    let $Q_{i(p-1)}:H_{n}(\Sigma_g; \F_p) \rightarrow H_{i(p-1)+pn}(\Sigma_{pg}; \F_p) $ be the homology operation determined by the image of $e_{i(p-1)} \in H_{i(p-1)}(C_p; \F_p)$ defined by
    $$x \in H_{n}(\Sigma_g; \F_p) \mapsto i_*(e_{i(p-1)} \otimes x^{\otimes p}) \in H_{i(p-1)+pn}(\Sigma_{pg}; \F_p),$$ where
    $i_*$ is the map induced by the inclusion $\Sigma_p \wr \Sigma_g \leq \Sigma_{pg}$. Similarly, for $i \geq 1$, we let $\beta Q_{i(p-1)}$ denote the homology operation determined by $e_{i(p-1)-1} \in H_{i(p-1)-1}(C_p ; \F_p)$.
\end{Def}

Because $H_i(C_p) \neq 0$ for all $i \geq 0$, but the same is not true for $H_i(\Sigma_p)$, it is often more convenient to use upper indexing for the Dyer--Lashof operations.

\begin{Def}\label{upperindexDLDef}
    Let $Q^s: H_n(\Sigma_g; \F_p) \rightarrow H_{n + 2s(p-1)}(\Sigma_{pg}; \F_p)$ be given by 
    $$Q^s = (-1)^{s}v(n) Q_{(2s-n)(p-1)},$$
    where $v(n)$ is from \ref{vEqu}. Note that if $2s - n < 0$, $Q_{(2s-n)(p-1)}$ is defined to be $0$. Similarly, we let
    $\beta Q^s: H_n(\Sigma_g; \F_p) \rightarrow H_{n + 2s(p-1)-1}(\Sigma_{pg}; \F_p)$ be given by
    $$\beta Q^s = (-1)^{s}v(n) \beta Q_{(2s-n)(p-1)},$$
    where $ \beta Q_{(2s-n)(p-1)} = 0$ if $2s - n \leq 0$.
\end{Def}

With these operations established, we can describe an $\F_p$-basis for the homology of symmetric groups.
Consider a sequence
$$I = (\epsilon_1, s_1, \epsilon_2, s_2, \dots, \epsilon_k, s_k)$$
where $\epsilon_i$ is either $0$ or $1$ and $s_i$ is in $\mathbb{Z}$. Any such $I$ determines a word in the Dyer--Lashof operations
$$Q^{I} := \beta^{\epsilon_1} Q^{s_1} \circ \dots \circ \beta^{\epsilon_k} Q^{s_k}.$$
We say such a sequence $I$ is admissible if both of the following conditions hold:
\begin{enumerate}
\item $s_i \in \mathbb{Z}$ for all $1 \leq i \leq k$.
\item $ps_i - \epsilon_i \geq s_{i-1}$ for all $1 < i \leq k$.
\end{enumerate}
We call $l(I) = k$, the length of $I$, and $e(I) : = 2s_1 - \epsilon_1 - \sum_{i = 2}^{k}(2s_i(p-1) - \epsilon_i)$, the excess of $I$. Further, we say a sequence $I$ is strongly admissible if $e(I) + \epsilon_1 > 0$ and $I$ is admissible.

Then, the set of strongly admissible sequences gives a basis for the homology of symmetric groups.
\begin{Thm}[{\cite[Thm 7.13]{Bernard}}]
\label{DLThm}
$H_*(\Sigma_{g}, \F_p)$ is isomorphic to the free commutative algebra on $\{Q^{I}x_0\}$ over $\F_p$, where $x_0 \in H_0(\Sigma_1, \F_p)$, $Q^{I}x_0 \in H_{t}(\Sigma_{p^{l(I)}}, \F_p)$ for $$t = \sum_{i = 1}^{k}(2s_i(p-1) - \epsilon_i),\;\; \sum_{I} p^{l(I)} = g,$$
and $Q^{I}$ ranges over all strongly admissible sequences.
\end{Thm}

\begin{Rem}\label{ImportantRem}
    For our purposes, we examine the basis of $H_*(\Sigma_{p^2})$. It contains elements of the form $\beta^{\epsilon} Q^r \circ Q^s$ for $s(p-1) < r \leq sp$. Note that here we omit writing $x_0$. By definition, 
    $$\beta^{\epsilon} Q^r \circ Q^s = (-1)^{r+s} v\big(2s(p-1)\big)v(0) \ \beta^{\epsilon}Q_{2\big(r-s(p-1)\big)(p-1)} \circ Q_{2s(p-1)},$$ 
    where $\beta ^{\epsilon}Q_{2\big(r-s(p-1)\big)(p-1)} \circ Q_{2s(p-1)} = i_* \big( x(r,s ; \epsilon) \big)$ for 
    $$x(r,s; \epsilon) := e_{2\big(r-s(p-1)\big)(p-1) -\epsilon} \otimes e_{2s(p-1)}^{\otimes p},$$
    and $i_*: H_*(\Sigma_p \wr \Sigma_p) \rightarrow H_*(\Sigma_{p^2})$.
    We define
    $$\beta^{\epsilon} Q^r \wr Q^s : = (-1)^{r+s} v\big(2s(p-1)\big)v(0) \ e_{2\big(r-s(p-1)\big)(p-1) -\epsilon} \otimes e_{2s(p-1)}^{\otimes p},$$
    where we identify the generators in $H_*(C_p)$ with the generators in $H_*(\Sigma_p)$. Since both $\beta^{\epsilon} Q^r \circ Q^s$ and $\beta^{\epsilon} Q^r \wr Q^s$ share the same normalization factor $\nu_{r,s} :=(-1)^{r+s} v\big(2s(p-1)\big)v(0)$, we will only focus on $x(r,s; \epsilon)$ in the computation.
\end{Rem}

\subsection{The conjugation maps}
For the double-coset formula in Corollary \ref{doublecosetCor}, the final missing ingredient is a concrete formula for each conjugation map
$$(c_{g^{-1}})_*: H_*(C_p \times C_p) \rightarrow H_*(C_p \times C_p)$$
where $g \in \{e, \gamma, \tau^{\lambda}\}$ for $\lambda \in \F_p^{\times}$. Note that since $\gamma^{-1} = \gamma$ and $(\tau^{\lambda})^{-1} = \tau^{-\lambda}$, and the map $\lambda \mapsto -\lambda$ simply permutes $\F_p^{\times}$, replacing $g$ with $g^{-1}$ merely reindexes the shear summands. Hence, the total Cartan sum remains unchanged.

To begin, the conjugation map induced by the identity is the identity map; it sends $e_{n} \otimes e_{m} \in H_*(C_p \times C_p)$ to itself.

Recall the generators $\alpha$ and $\beta$ from Subsection \ref{subgrouptowersubsection}. Then, for the conjugation map induced by the swap element, we deploy a result from May.
\begin{Lem}[{\cite[Lem 4.4]{May}}]
\label{swapLem}
    Suppose $c_{\gamma}: C_p \times C_p \rightarrow C_p \times C_p$ is a conjugation map sending $\alpha$ to $\beta$ and $\beta$ to $\alpha$. Then the induced map $(c_{\gamma})_*: H_*(C_p \times C_p) \rightarrow H_*(C_p \times C_p)$ is given by
    $$(c_{\gamma})_*(e_n \otimes e_m) = (-1)^{nm} e_{m} \otimes e_{n}.$$
\end{Lem} 

Finally, for the conjugation map induced by the shear element, we divide the statement into cases based on parity. We first recall a standard fact regarding the cohomology ring structure of $C_p$ over $\F_p$.

\begin{Lem}[{\cite[Chapter 12, Section 7]{CartanEilenberg}}]
\label{CohomologyCpLem}
    $H^*(C_p; \F_p) \cong \wedge[u] \otimes \F_p[v]$ where $|u| = 1$ and $v = \beta (u)$.
\end{Lem}

For $i \in \{1, 2\}$, let $p_i: C_p \times C_p \rightarrow C_p$ be the projection map onto the $i$-th coordinate.  We then have $H^*(C_p \times C_p) \cong \wedge[u_1,u_2] \otimes \F_p[v_1,v_2]$ where $u_i = p_i^*(u)$ and $v_i = p_i^*(v)$.

\begin{Lem}\label{shearLem}
    Let $c_{\tau^{\lambda}}: C_p \times C_p \rightarrow C_p \times C_p$ be the conjugation map that sends $\alpha \mapsto \alpha$ and $\beta \mapsto \alpha^{\lambda} \beta$ for $\lambda \in \{1, 2, \dots, p-1\}$. Then the induced map $(c_{\tau^{\lambda}})_*: H_*(C_p \times C_p) \rightarrow H_*(C_p \times C_p)$ acts as follows:
    \begin{enumerate}
        \item $e_{2n} \otimes e_{2m} \mapsto \sum_{k=0}^m \binom{n+k}{k} \lambda^k e_{2(n+k)} \otimes e_{2(m-k)}$;
        \item $e_{2n+1} \otimes e_{2m} \mapsto \sum_{k=0}^m \binom{n+k}{k} \lambda^k e_{2(n+k)+1} \otimes e_{2(m-k)}$;
        \item $e_{2n+1} \otimes e_{2m+1} \mapsto \sum_{k=0}^m \binom{n+k}{k} \lambda^k e_{2(n+k)+1} \otimes e_{2(m-k)+1}$; and
        \item $e_{2n} \otimes e_{2m+1} \mapsto \sum_{k=0}^m \binom{n+k}{k} \lambda^k e_{2(n+k)} \otimes e_{2(m-k)+1} + \sum_{k=0}^{m} \binom{n+k}{k} \lambda^{k+1} e_{2(n+k)+1} \otimes e_{2(m-k)}.$
    \end{enumerate}
\end{Lem}

\begin{proof}
    Let $\phi = c_{\tau^{\lambda}}$. We first show that $\phi^*: H^*(C_p \times C_p) \rightarrow H^*(C_p \times C_p)$ sends
    \begin{enumerate}
        \item $v_1 \mapsto v_1 + \lambda v_2$;
        \item $v_2 \mapsto v_2$;
        \item $u_1 \mapsto u_1 + \lambda u_2$; and
        \item $u_2 \mapsto u_2$.
    \end{enumerate}
    Since $p_2 \circ \phi = p_2$, we have $\phi^*(u_2) = \phi^*\circ p_2^* (u) = p_2^* (u) = u_2$. 
    Define $\nabla = p_1 \circ \phi: C_p \times C_p \rightarrow C_p$, then $\nabla^*:H^*(C_p) \rightarrow H^*(C_p \times C_p)$ sends $u$ to $c_1 u_1 + c_2u_2$. Since $\nabla \circ i_1 = \id$, $p_1 \circ i_1 = \id$ and $p_2 \circ i_1 = 0$, we have $u = i_1^* \circ \nabla^* (u) = c_1 i_1^* (u_1) + c_2 i_1^* (u_2) = c_1i_1^* \circ p_1^* (u) + c_2 i_1^* \circ p_2^* (u) = c_1u$, which means $c_1 = 1$. Similarly, by $\nabla \circ i_2 = \lambda \cdot \id$, $p_2 \circ i_2 = \id$ and $p_1 \circ i_2 = 0$, we have $c_2 = \lambda$. Hence, $\phi^*(u_1) = u_1 + \lambda u_2.$ Now, we have $\phi^*(v_1) = \phi^*(\beta(u_1)) = \beta (\phi^*(u_1)) = \beta (u_1 + \lambda u_2 )=v_1 + \lambda v_2$. Similarly, $\phi^*(v_2) = v_2$. 
    
    Next, we prove the cases for $e_{2n} \otimes e_{2m}$ and $e_{2n} \otimes e_{2m+1}$. The other two cases are similar. Suppose $$\phi_*(e_{2n} \otimes e_{2m}) = \sum_{p,q}c_{p,q}e_{2p} \otimes e_{2q} + \sum_{i,j}d_{i,j} e_{2i+1} \otimes e_{2j+1}.$$ We shall see that $d_{i,j} = 0$. For $e_{2i+1} \otimes e_{2j+1}$, its dual is $u_1v_1^iu_2v_2^j$. By adjunction, $$d_{i,j} = \langle u_1v_1^iu_2v_2^j, \phi_*(e_{2n} \otimes e_{2m}) \rangle = \langle \phi^*(u_1v_1^iu_2v_2^j), e_{2n} \otimes e_{2m} \rangle.$$ 
    Furthermore, since $u^2 = 0$, we have that 
    \begin{align*}
        \phi^*(u_1v_1^iu_2v_2^j) & = \phi^*(u_1u_2v_1^iv_2^j) \\
        & = \phi^*(u_1)\phi^*(u_2)\phi^*(v_1)^i\phi^*(v_2)^j \\
        & = (u_1 + \lambda u_2)u_2 \phi^*(v_1)^i\phi^*(v_2)^j \\
        & = u_1 u_2 \phi^*(v_1)^i\phi^*(v_2)^j.
    \end{align*}
    Since the monomial contains the dual basis $u_1u_2$, whereas the dual of $e_{2n} \otimes e_{2m}$ is $v_1^nv_2^m$, it follows that $d_{i,j} = 0$. 
    Next, we evaluate:
    $$\begin{aligned} \phi^*(v_1^p v_2^q) &= \phi^*(v_1)^p \phi^*(v_2)^q \\ &= (v_1 + \lambda v_2)^p v_2^q \\ &= \sum_{k=0}^p \binom{p}{k} \lambda^k v_1^{p-k} v_2^{q+k}. \end{aligned}$$
    The coefficient $c_{p,q} = \langle \phi^*(v_1^p v_2^q), e_{2n} \otimes e_{2m} \rangle$ is exactly the coefficient of $v_1^n v_2^m$ in the expansion of $\phi^*(v_1^p v_2^q)$. Performing the index substitution $p-k = n$ and $q+k = m$, we find $c_{p,q} = c_{n+k, m-k} = \binom{n+k}{k} \lambda^k$. 
    Therefore:
    $$\phi_*(e_{2n} \otimes e_{2m}) = \sum_{k = 0}^m \binom{n+k}{k} \lambda^k e_{2(n+k)} \otimes e_{2(m-k)}.$$
    Now, suppose 
    $$\phi_*(e_{2n} \otimes e_{2m+1}) = \sum_{p,q}c_{p,q} e_{2p} \otimes e_{2q+1} + \sum_{i,j} d_{i,j} e_{2i+1} \otimes e_{2j}.$$
    A similar argument shows that 
    $c_{p,q} = c_{n+k,m-k} = \binom{n+k}{k} \lambda^k$. 
    As for $d_{i,j}$, we compute:
    $$\phi^*(u_1v_1^iv_2^j) = (u_1 + \lambda u_2)(v_1+ \lambda v_2)^iv_2^j = u_1(v_1+\lambda v_2)^iv_2^j + \lambda u_2 (v_1+\lambda v_2)^iv_2^j.$$ The $u_1$ summand cannot pair with $v_1^nu_2v_2^m$, the dual of $e_{2n} \otimes e_{2m+1}$, so the coefficient for this summand shall be $0$. The $\lambda u_2$ summand contributes the coefficient $\binom{n+k}{k} \lambda^{k+1}$. Together, this shows that 
    $$\phi_*(e_{2n} \otimes e_{2m+1}) = \sum_{k=0}^m \binom{n+k}{k} \lambda^k e_{2(n+k)} \otimes e_{2(m-k)+1} + \sum_{k=0}^{m} \binom{n+k}{k} \lambda^{k+1} e_{2(n+k)+1} \otimes e_{2(m-k)}.$$
\end{proof}

\section{The wreath transfer formula for odd primes}
In this section, we prove the main theorem of the paper. Fix integers $s$ and $r$ such that $s(p-1) < r \leq sp$. Equivalently, we may write $r = Ps + d$ for $P = p-1$ and $1 \leq d \leq s$.

\subsection{The preimage element \texorpdfstring{$X(r,s;\epsilon)$}{X(r,s; epsilon)}}
By Remark \ref{ImportantRem}, the basis element $\beta^{\epsilon} Q^r \circ Q^s \in H_*(\Sigma_{p^2})$ is defined as 
$\beta^{\epsilon} Q^r \circ Q^s = \nu_{r,s} \ i_*(x(r,s; \epsilon)),$
where $$x(r,s; \epsilon) = e_{2(r-sP)P-\epsilon} \otimes e_{2sP}^{\otimes p},$$ and
$\nu_{r,s} = (-1)^{r+s}v(2sP)v(0)$ with $v(0) = 1$. Observe that every residual class $\beta^{\epsilon} Q^{r+s-j} \wr Q^j$ shares this same normalization factor, since $r+s-j + j = r+s$ and $v(2jP) \equiv v(2sP) \equiv 1$. Thus, the common scalar may be suppressed during the coefficient calculation and restored at the end.

Recall from Corollary \ref{doublecosetCor}, we want to find an input $X = X(r, s ; \epsilon) \in H_*(C_p \times C_p)$ such that $(\iota_{\Delta})_*(X(r,s; \epsilon)) = x(r,s; \epsilon).$
The following lemma establishes the existence and explicit form of this element.

\begin{Lem}\label{preimageLem}
Put $P=p-1$. There is a unique choice of coefficients
$a_0,\dots,a_{s-1}\in\F_p$ such that the element
$$X(r,s;0)= \sum_{k=0}^{s-1} a_k \Bigl(e_{2(r+k)P}\otimes e_{2(s-k)P}\Bigr)
$$
satisfies
$(\iota_{\Delta})_*\bigl(X(r,s;0)\bigr) =x(r,s;0).$
The coefficients $a_k$ are the unique solution of the triangular system
$$\sum_{k=0}^{n} a_k\, t\bigl(2(s-k)P,n-k\bigr) = \delta_{n,0},
\qquad 0\leq n\leq s-1,$$
and they are given by $$ a_k = (-1)^k \binom{-sP-1}{k}, \qquad
0\leq k\leq s-1.$$

If $\beta$ denotes the homological Bockstein, define
$ X(r,s;1):=\beta X(r,s;0).$
Explicitly, $$ X(r,s;1) = \sum_{k=0}^{s-1} a_k \Bigl(
e_{2(r+k)P-1}\otimes e_{2(s-k)P} +
e_{2(r+k)P}\otimes e_{2(s-k)P-1} \Bigr).$$
Then $ (\iota_{\Delta})_*\bigl(X(r,s;1)\bigr) = x(r,s;1).$
\end{Lem}

\begin{proof}
By Lemma \ref{Mayformula}, we have
$$
(\iota_{\Delta})_* \bigl(X(r,s;0)\bigr)
=
\sum_{k=0}^{s-1}
\sum_{h\geq 0}
a_k\,
t\bigl(2(s-k)P,h\bigr)
e_{2(r-sP+p(k+h))P}
\otimes
e_{2(s-k-h)P}^{\otimes p}.
$$
Set $n=k+h.$
Collecting all terms having the same value of $n$ gives
$$
(\iota_{\Delta})_*\bigl(X(r,s;0)\bigr)
=
\sum_{n=0}^{s-1}
\left(
\sum_{k=0}^{n}
a_k\,
t\bigl(2(s-k)P,n-k\bigr)
\right)
e_{2(r-sP+pn)P}
\otimes
e_{2(s-n)P}^{\otimes p}.
$$
The term corresponding to $n=s$ vanishes, because for
$0\leq k\leq s-1$,
$$
t\bigl(2(s-k)P,s-k\bigr)
=
(-1)^{s-k}
\binom{0}{s-k}
=
0.
$$
Terms with $n>s$ vanish because the second lower index is negative.

Since $x(r,s;0) = e_{2(r-sP)P}\otimes e_{2sP}^{\otimes p}$,
we have
$(\iota_{\Delta})_*\bigl(X(r,s;0)\bigr) = x(r,s;0)$
if and only if
$$
\sum_{k=0}^{n}
a_k\,
t\bigl(2(s-k)P,n-k\bigr)
=
\delta_{n,0},
\qquad
0\leq n\leq s-1.
$$

We now verify that
$$
a_k=(-1)^k\binom{-sP-1}{k}
$$
satisfies this system. Since $P=p-1$ is even, we have
$$
\begin{aligned}
\sum_{k=0}^{n}
a_k\,
t\bigl(2(s-k)P,n-k\bigr)
&=
(-1)^n
\sum_{k=0}^{n}
\binom{-sP-1}{k}
\binom{(s-n)P}{n-k}
\\
&=
(-1)^n
\binom{-nP-1}{n}
\\
&=
\binom{nP+n}{n}
\\
&=
\binom{np}{n}.
\end{aligned}
$$
The second equality follows from Vandermonde's identity, and the third
uses the generalized binomial identity.

If $n=0$, then
$\binom{np}{n}=1$.
If $n>0$, then by Lucas' theorem in $\F_p$,
$\binom{np}{n} = 0$.
Thus
$$
\sum_{k=0}^{n}
a_k\,
t\bigl(2(s-k)P,n-k\bigr)
=
\delta_{n,0}.
$$
Note that the system is lower triangular, and its diagonal entries are
$t\bigl(2(s-n)P,0\bigr)=1.$
Therefore the coefficients $a_0,\dots,a_{s-1}$ are uniquely determined.

For the case $\epsilon=1$, define
$X(r,s;1)=\beta X(r,s;0)$.
The homological Bockstein satisfies the Cartan formula
$$\beta(a\otimes b)
=
\beta(a)\otimes b
+
(-1)^{|a|}a\otimes\beta(b).$$
Since $e_{2(r+k)P}$ has an even degree, this proves the displayed formula for $X(r,s;1)$.

Finally, by naturality of the Bockstein,
$$
(\iota_{\Delta})_*\bigl(X(r,s;1)\bigr)
= (\iota_{\Delta})_* \bigl(\beta X(r,s;0)\bigr)=
\beta (\iota_{\Delta})_* \bigl(X(r,s;0)\bigr)
=
\beta x(r,s;0).
$$
Applying the Cartan formula gives
$$
\begin{aligned}
\beta x(r,s;0)
&=
e_{2(r-sP)P-1}
\otimes
e_{2sP}^{\otimes p}
\\
&\quad+
e_{2(r-sP)P}
\otimes
\beta\bigl(e_{2sP}^{\otimes p}\bigr).
\end{aligned}
$$
The second summand vanishes. Indeed, if $s=0$, then $\beta\bigl(e_0^{\otimes p}\bigr)=0$. If $s\geq1$, the Cartan formula gives
$$\beta\left(e_{2sP}^{\otimes p}\right)
=
\sum_{a=1}^{p}
e_{2sP}^{\otimes(a-1)}
\otimes e_{2sP-1}
\otimes e_{2sP}^{\otimes(p-a)}.$$
In each summand, the inner tensor factors are not purely identical. By the basis description in \eqref{wrbasis}, a class with nonconstant inner factors can occur only when the outer factor is $e_0$. 
On the other hand, since $r>sP$, we have $e_{2(r-sP)P}\neq e_0$.
Therefore,
$$
\beta x(r,s;0)
=
e_{2(r-sP)P-1}
\otimes
e_{2sP}^{\otimes p}
=
x(r,s;1).
$$
Thus
$(\iota_{\Delta})_*\bigl(X(r,s;1)\bigr) = x(r,s;1)$
as required.
\end{proof}

\subsection{Support lemma and coefficient decomposition}
\label{subsec:support-coefficients}
Let $S = \{e,\gamma, \tau^\lambda\mid\lambda\in\F_p^\times\}$.
For the element $X(r,s; \epsilon)$ constructed in Lemma \ref{preimageLem}, define the total
double-coset sum
$$
\mathcal D(r,s; \epsilon)
=
\sum_{g\in  S}
i_*\circ(\iota_\Delta)_*\circ(c_{g^{-1}})_*
\bigl(X(r,s;\epsilon)\bigr).
$$
By Corollary~\ref{doublecosetCor} and the linearity of the transfer, we have
$$
\tr\bigl(\beta^\epsilon Q^r\circ Q^s\bigr)
=
\nu_{r,s}\mathcal D(r,s; \epsilon).
$$
Thus we may compute the coefficients in $\mathcal{D}(r,s; \epsilon)$ and
restore the common factor $\nu_{r,s}$ at the end.

There is a support lemma to determine $\mathcal{D}(r,s; \epsilon)$ in terms of the basis elements described in \ref{wrbasis}.

\begin{Lem}[Support lemma]
\label{supportLem}
For $\epsilon=0$, only tensors of the form
$$
e_{2(r+s-pj)P}\otimes e_{2jP}^{\otimes p}
$$
can occur in $\mathcal D(r,s;0)$ for $0\leq j\leq
\left\lfloor\frac{r+s}{p}\right\rfloor$, representing the classes
$Q^{r+s-j}\wr Q^j$.

For $\epsilon=1$, only tensors of the following two forms can occur in
$\mathcal D(r,s;1)$. The first family is
$$
e_{2(r+s-pj)P-1}\otimes e_{2jP}^{\otimes p}
$$
representing the classes
$\beta Q^{r+s-j}\wr Q^j$ for $0\leq j\leq
\left\lfloor\frac{r+s-1}{p}\right\rfloor$. The second family is
$$
e_{(2(r+s-pj)+1)P}\otimes e_{2jP-1}^{\otimes p}
$$
representing the classes
$Q^{r+s-j}\wr\beta Q^j$ for $1\leq j\leq
\left\lfloor\frac{r+s}{p}\right\rfloor$.
No mixed wreath tensors occur.
\end{Lem}

\begin{proof}
Each conjugation map $(c_{g^{-1}})_*$ sends a basis element
$e_a\otimes e_b$ to a linear combination of basis elements of the same
form. The diagonal-embedding formula in Lemma~\ref{Mayformula} then sends each such element to a linear combination of pure tensors of the form:
$$e_u\otimes e_v^{\otimes p}.$$
Hence, no mixed wreath tensors occur.

For $\epsilon=0$, the two indices are even. Write
$v=2jP$.
Since the total degree of $x(r,s; 0)$ is $2(r+s)P$, the total degree of $e_u\otimes e_v^{\otimes p}$ must be $2(r+s)P$, so
$u=2(r+s-pj)P$.
Therefore, the nonnegativity of $u$ and $v$ gives
$$
0\leq j\leq
\left\lfloor\frac{r+s}{p}\right\rfloor.
$$

For $\epsilon=1$, the total degree is $2(r+s)P-1$.
There are two possible parity patterns.

First suppose that the first tensor factor has odd degree and the second
tensor factor has even degree. Then we may write
$v=2jP$
with $j\geq 0$. The total degree condition gives
$u+pv=2(r+s)P-1$,
so
$u=2(r+s-pj)P-1$.
The condition $u\geq 0$ is equivalent to
$r+s-pj\geq 1$.
Therefore this family occurs only for
$$
0\leq j\leq
\left\lfloor\frac{r+s-1}{p}\right\rfloor.
$$

Now suppose that the first tensor factor has even degree and the second
tensor factor has odd degree. Then we may write
$v=2jP-1$
with $j\geq 1$. The total degree condition gives
$u+p(2jP-1)=2(r+s)P-1$,
and hence $u=(2(r+s-pj)+1)P$.
The condition $u\geq 0$ is equivalent to
$r+s-pj\geq 0$.
Therefore this family occurs only for
$$
1\leq j\leq
\left\lfloor\frac{r+s}{p}\right\rfloor.
$$
Together, this gives the two stated families.
\end{proof}

We first consider the case $\epsilon=0$. For each supported value of $j$ from Lemma \ref{supportLem},
let $C_j$ denote the coefficient of
$$
e_{2(r+s-pj)P}\otimes e_{2jP}^{\otimes p}
$$
in $\mathcal D(r,s; 0)$. We write
$$
C_j=I_j+S_j+G_j,
$$
where $I_j$, $S_j$, and $G_j$ denote the contributions of the identity,
the shear elements, and the swap element, respectively.

For $q\in\mathbb Z$, define
$$
c_j(q)
=
\begin{cases}
(-1)^q\binom{jP}{q}, & q\geq0,\\
0, & q<0.
\end{cases}
$$

Now, we can write each $I_j$, $S_j$ and $G_j$ as a linear combination of the coefficients $a_k$ and $c_j(q)$.

\begin{Prop}[Coefficient decomposition]
\label{coefficientDecompositionProp}
Assume $\epsilon=0$. For each integer $0 \leq j \leq \lfloor (r+s)/p \rfloor$, we have
$$I_j=\delta_{j,s},$$
$$
S_j
=
-
\sum_{\substack{0\leq k\leq s-1,\ h\geq0\\k+h\leq s-j}}
a_k
\binom{(r+k+h)P}{hP}
c_j(s-j-k-h),
$$
and
$$
G_j
=
\sum_{k=0}^{s-1}
a_kc_j(r+k-j).
$$
Empty sums are interpreted as zero.
\end{Prop}

\begin{proof}
The identity contribution is $(\iota_\Delta)_*X(r,s;0)=x(r,s;0)$,
which is precisely the target tensor corresponding to $j=s$. Hence $I_j=\delta_{j,s}$.

We next consider the shear elements. Since
$
(\tau^\lambda)^{-1}=\tau^{-\lambda}
$
and $\lambda\mapsto-\lambda$ permutes $\F_p^\times$, using the inverse
only reindexes the total shear sum. For a fixed $k$ and $\lambda$, the shear formula from Lemma \ref{shearLem} sends the $k$-th summand $a_k \Bigl(e_{2(r+k)P}\otimes e_{2(s-k)P}\Bigr)$
of $X(r,s;0)$ to
$$
a_k
\sum_{l=0}^{(s-k)P}
\binom{(r+k)P+l}{l}
\lambda^l
e_{2((r+k)P+l)}
\otimes
e_{2((s-k)P-l)}.
$$
When summing over $\lambda\in\F_p^\times$, we use the standard identity:
$$
\sum_{\lambda\in\F_p^\times}\lambda^l
=
\begin{cases}
-1, & P\mid l,\\
0, & P\nmid l.
\end{cases}
$$
Thus only the terms with $l=hP$
survive, and their total contribution is
$$
-
a_k
\binom{(r+k+h)P}{hP}
e_{2(r+k+h)P}
\otimes
e_{2(s-k-h)P}.
$$
Applying the diagonal-embedding formula in Lemma~\ref{Mayformula} to
$$
e_{2(r+k+h)P}\otimes e_{2(s-k-h)P},
$$
a summand indexed by $q\geq 0$ has coefficient
$t(2(s-k-h)P,q)$ and is of the form
$$
e_{2(r+k+h)P+(2pq-2(s-k-h)P)P}
\otimes
e_{2(s-k-h-q)P}^{\otimes p}.
$$
This term contributes to $C_j$ only when the second tensor factor is
$e_{2jP}^{\otimes p}$, that is, when $q=s-j-k-h$.
For this value of $q$, the first tensor factor becomes $e_{2(r+s-pj)P}$.
Moreover,
$$
\begin{aligned}
t\bigl(2(s-k-h)P,q\bigr)
&=
(-1)^q\binom{(s-k-h-q)P}{q}  \\
&=
(-1)^q\binom{jP}{q}
=
c_j(q).
\end{aligned}
$$
Therefore
$$
S_j
=
-
\sum_{\substack{0\leq k\leq s-1,\ h\geq0\\k+h\leq s-j}}
a_k
\binom{(r+k+h)P}{hP}
c_j(s-j-k-h).
$$

Finally, consider the swap element. By the swap formula from Lemma~\ref{swapLem}, it sends the $k$-th summand of $X(r,s;0)$ to
$$
a_k
e_{2(s-k)P}\otimes e_{2(r+k)P}.
$$
Applying the diagonal-embedding formula in Lemma~\ref{Mayformula} to
$$
e_{2(s-k)P}\otimes e_{2(r+k)P},
$$
a summand indexed by $q\geq 0$ has coefficient
$t(2(r+k)P,q)$ and is of the form
$$
e_{2(s-k)P+(2pq-2(r+k)P)P}
\otimes
e_{2(r+k-q)P}^{\otimes p}.
$$
This term contributes to $C_j$ only when the second tensor factor is
$e_{2jP}^{\otimes p}$, that is, when $q=r+k-j$.
For this value of $q$, the first tensor factor becomes $e_{2(r+s-pj)P}$.
Moreover,
$$
\begin{aligned}
t\bigl(2(r+k)P,q\bigr)
&=
(-1)^q\binom{(r+k-q)P}{q}  \\
&=
(-1)^q\binom{jP}{q}
=
c_j(q).
\end{aligned}
$$
Hence
$$
G_j
=
\sum_{k=0}^{s-1}
a_kc_j(r+k-j).
$$
\end{proof}

\subsection{Evaluation of the coefficients}
\label{subsec:evaluation-coefficients}
We have expressed these local contributions as linear combinations of the coefficients $a_k$ and $c_j(q)$.
Continuing with the case $\epsilon=0$, we evaluate the total contribution $C_j$ in the two ranges $0 \leq j < s$ and $s \leq j \leq \lfloor (r+s)/p \rfloor$.

The proof of the following statement is deferred to
Section~\ref{shearswapcollapsesection}. 

\begin{Prop}
\label{interiorCoefficientProp}
For $0\leq j<s$, we have the total contribution
$$
C_j
= (-1)^{r-j}
\binom{(j-s)(p-1)-1}{r-(p-1)s-j-1}.
$$
\end{Prop}

\begin{Prop}\label{boundaryCoefficientProp}
For
$
s\leq j\leq
\left\lfloor\frac{r+s}{p}\right\rfloor,
$
we have the total contribution
$$
C_j
=
\begin{cases}
1, & j=s,\\
0, & j>s.
\end{cases}
$$
\end{Prop}

\begin{proof}
First, let $j=s$. By Proposition~\ref{coefficientDecompositionProp}, $I_s = 1$. The condition $k+h \leq s-j = 0$ forces $k=h=0$, and hence $$S_s = -a_0\binom{rP}{0}c_s(0) = -1.$$

For the swap contribution, using the definition of $a_k$ and $c_j$, we have
$$
\begin{aligned}
G_s
&=
\sum_{k=0}^{s-1}a_kc_s(r+k-s)\\
&=
(-1)^{r-s}
\sum_{k=0}^{s-1}
\binom{-sP-1}{k}
\binom{sP}{r+k-s}.
\end{aligned}
$$
Since $r=Ps+d$ for $1 \leq d \leq s$,
$$
\binom{sP}{r+k-s}
=
\binom{sP}{s-d-k}.
$$
Because $s-d\leq s-1$, extending the sum to all $k\geq0$ adds only
zero terms. Therefore, by Vandermonde's identity,
$$
\begin{aligned}
G_s
&=
(-1)^{r-s}
\sum_{k\geq0}
\binom{-sP-1}{k}
\binom{sP}{s-d-k}\\
&=
(-1)^{r-s}\binom{-1}{s-d}\\
&=
(-1)^{r-d}\\
&=
1,
\end{aligned}
$$
because $r-d=Ps$ and $P$ is even. Thus $C_s=I_s+S_s+G_s=1-1+1=1.$

Now suppose $s < j \leq \lfloor (r+s)/p \rfloor$. Then $I_j = 0$. Moreover, the condition $k+h \leq s-j$ has no solutions with $k,h \geq 0$, so $S_j = 0$.

It remains to consider $G_j$. Let $M = pj - r$. The supported range gives $0 < M \leq s$. Using the definitions of $a_k$ and $c_j$, we obtain
$$G_j = (-1)^{r-j} \sum_{k=0}^{s-1} \binom{-sP-1}{k} \binom{jP}{M-k}.$$

If $M < s$, extending the sum to all $k \geq 0$ adds only zero terms. If $M = s$, the only additional possible term vanishes in $\F_p$ by Lucas' theorem:
$$\binom{-sP-1}{s} = (-1)^s\binom{sp}{s} = 0.$$
Hence, in either case, Vandermonde's identity gives
$$G_j = (-1)^{r-j} \binom{(j-s)P-1}{M}.$$

Since $M - \bigl((j-s)P-1\bigr) = j-d+1 > 0$, the lower index is larger than the nonnegative upper index. Therefore, $G_j = 0$. It follows that $C_j = I_j + S_j + G_j = 0$ for every $j > s$ in the supported range.
\end{proof}

\subsection{The Bockstein case}
\label{subsec:bockstein-case}

We now pass from the case $\epsilon=0$ to the case $\epsilon=1$.

\begin{Prop}
\label{bocksteinCoefficientProp}
For each $0\leq j\leq \lfloor(r+s)/p\rfloor$, the coefficient of $\beta Q^{r+s-j}\wr Q^j$ in $\mathcal D(r,s;1)$ is the same coefficient $C_j$ computed for $\mathcal D(r,s;0)$. Moreover, the coefficient of $Q^{r+s-j}\wr\beta Q^j$ is zero.
\end{Prop}

\begin{proof}
By Lemma~\ref{preimageLem}, $X(r,s;1)=\beta X(r,s;0)$. The Bockstein commutes with all homomorphisms induced by conjugation and inclusion. Therefore, $\mathcal{D}(r,s;1) = \beta\mathcal D(r,s;0)$.

By Propositions~\ref{interiorCoefficientProp} and
\ref{boundaryCoefficientProp},
$$
\mathcal D(r,s;0)
=
\sum_j
C_j\,
e_{2(r+s-pj)P}
\otimes
e_{2jP}^{\otimes p},
$$
where the sum is taken over the supported values of $j$ in Lemma \ref{supportLem}.

Since both lower indices are even, the Cartan formula for the Bockstein
gives
$$
\begin{aligned}
\beta\left(
e_{2(r+s-pj)P}
\otimes
e_{2jP}^{\otimes p}
\right)
&=
e_{2(r+s-pj)P-1}
\otimes
e_{2jP}^{\otimes p}
\\
&\quad+
e_{2(r+s-pj)P}
\otimes
\beta\left(e_{2jP}^{\otimes p}\right).
\end{aligned}
$$
The second summand vanishes. Indeed, if $j=0$, then $\beta\bigl(e_0^{\otimes p}\bigr)=0$. If $j\geq1$, the Cartan formula gives$$\beta\left(e_{2jP}^{\otimes p}\right)
=
\sum_{a=1}^{p}
e_{2jP}^{\otimes(a-1)}
\otimes e_{2jP-1}
\otimes e_{2jP}^{\otimes(p-a)}.$$
In each summand, the inner tensor factors are not purely identical. By the basis description in \eqref{wrbasis}, a class with nonconstant inner factors can occur only when the outer factor is $e_0$.
On the other hand, by Propositions \ref{interiorCoefficientProp} and \ref{boundaryCoefficientProp}, $C_j=0$ for $j>s$. Thus any term with $C_j\neq 0$ has $j\leq s$, and for such a term
$$
2(r+s-pj)P \geq 2(r+s-ps)P = 2dP > 0.
$$
Hence the outer factor has positive degree, so $e_{2(r+s-pj)P} \otimes \beta\left(e_{2jP}^{\otimes p}\right) = 0$.
Therefore,
$$\beta\left( e_{2(r+s-pj)P}\otimes e_{2jP}^{\otimes p} \right) = e_{2(r+s-pj)P-1}\otimes e_{2jP}^{\otimes p}.$$
This is the lower-index representative of $\beta Q^{r+s-j}\wr Q^j$. Therefore, its coefficient is $C_j$, while no term of the form $Q^{r+s-j}\wr\beta Q^j$ occurs.
\end{proof}

\subsection{Proof of the wreath transfer formula for odd primes}
\label{subsec:proof-main-theorem}

In this subsection, we prove the main theorem of the paper and then derive Corollary \ref{KjaerCor}.

\begin{proof}[Proof of Theorem~\ref{mainThm}]
By Lemma~\ref{supportLem}, the total double-coset sum $\mathcal D(r,s; \epsilon)$ can only contain terms corresponding to $\beta^\epsilon Q^{r+s-j}\wr Q^j$ for $0\leq j\leq \lfloor(r+s)/p\rfloor$, together, when $\epsilon=1$, with possible inner-Bockstein terms $Q^{r+s-j}\wr\beta Q^j$.

For $\epsilon=0$, Propositions~\ref{interiorCoefficientProp} and
\ref{boundaryCoefficientProp} give
$$
C_j=
\begin{cases}
(-1)^{r-j}
\displaystyle
\binom{(j-s)(p-1)-1}{r-(p-1)s-j-1},
& 0\leq j<s,\\[10pt]
1,
& j=s,\\[4pt]
0,
& j>s.
\end{cases}
$$

For $\epsilon=1$, Proposition~\ref{bocksteinCoefficientProp} shows that the coefficients of $\beta Q^{r+s-j}\wr Q^j$ are the same $C_j$, and that all coefficients of $Q^{r+s-j}\wr\beta Q^j$ are zero.

It remains to restore the common normalization factor. By the definition of the upper-index wreath classes, the lower-index representative $e_{2(r+s-pj)(p-1)-\epsilon}\otimes e_{2j(p-1)}^{\otimes p}$ corresponds to $\beta^\epsilon Q^{r+s-j}\wr Q^j$ after multiplication by $\nu_{r+s-j,j}$. Since $(r+s-j)+j=r+s$ and $v(2j(p-1)) \equiv v(2s(p-1)) \equiv 1$, we have $\nu_{r+s-j,j}=\nu_{r,s}$.

Therefore, we have
$$
\begin{aligned}
    \tr\bigl(\beta^\epsilon Q^r\circ Q^s\bigr) 
    & = \nu_{r,s}\mathcal D(r,s; \epsilon)\\
    & = C_s \beta^\epsilon Q^r \wr Q^s + \sum_{0 \leq j < s} C_j \beta^\epsilon Q^{r+s-j} \wr Q^j\\
    & = \beta^\epsilon Q^r\wr Q^s
+
\sum_{0\leq j<s}
(-1)^{r-j}
\binom{(j-s)(p-1)-1}{r-(p-1)s-j-1}
\beta^\epsilon Q^{r+s-j}\wr Q^j,
\end{aligned}
$$
where the first equality is by the argument before Lemma \ref{supportLem} and the second equality is due to $\nu_{r+s-j,j}=\nu_{r,s}$.
\end{proof}

We then give a proof of Kjaer's conjecture.
\begin{proof}[Proof of Corollary~\ref{KjaerCor}]
By the Arone--Mahowald calculation, as used in \cite[Remark 4.4]{Kjaer}, there is a
surjection
$$
\pi:
\Sigma^{-2}H_*\bigl((\mathbb{S}^{2l+1})^{\wedge p^2}_{h(\Sigma_p\wr\Sigma_p)};\mathbb F_p\bigr)
\longrightarrow
H_*(\mathbb D_{p^2}(\mathbb{S}^{2l+1});\mathbb F_p)
$$
whose kernel, after identifications, is the image of the transfer
$$
\tr: H_*(\Sigma_{p^2}; \F_p) \longrightarrow H_*(\Sigma_p \wr \Sigma_p; \F_p).
$$
Now, we have Theorem \ref{mainThm} gives an element of $\im(\tr)$ of the
form
$$
\beta^\epsilon Q^r\wr Q^s
+
\sum_{0\leq j<s}
(-1)^{r-j}
\binom{(j-s)(p-1)-1}{r-(p-1)s-j-1}
\beta^\epsilon Q^{r+s-j}\wr Q^j.
$$
Applying $\pi$ gives zero. Therefore, in
$H_*(\mathbb D_{p^2}(\mathbb{S}^{2l+1});\mathbb F_p)$,
$$
\beta^\epsilon Q^rQ^s\iota
=
-\sum_{0\leq j<s}
(-1)^{r-j}
\binom{(j-s)(p-1)-1}{r-(p-1)s-j-1}
\beta^\epsilon Q^{r+s-j}Q^j\iota.
$$
\end{proof}

\section{The shear--swap coefficient cancellation}\label{shearswapcollapsesection}
The goal of this section is to prove Proposition \ref{interiorCoefficientProp}. In the double-coset calculation, the coefficient $C_j$ is a sum of three contributions:
$C_j=I_j+S_j+G_j.$
The main difficulty is the evaluation of the shear contribution $S_j$ and the swap contribution $G_j$ for $0 \leq j < s$, where the identity contribution $I_j$ vanishes. We handle this by a generating-function argument. 
The coefficient families appearing in the shear and swap terms can be represented by the series $A_s(z)$, $C_j(z)$, and their product $F_m(z)=(1-z)^{-m(p-1)-1}$. The shear term is then rewritten as
\[
S_j=-[z^m]F_m(z)D_{A_0}(z),
\]
where $D_{A_0}(z)$ is a $P$-section polynomial. A separate coefficient-extraction argument expresses the swap term as
\[
G_j=[z^m]F_m(z)(D_{A_0}(z)-z^{A_0}).
\]
Adding the two expressions cancels the common $D_{A_0}(z)$-term and leaves a single coefficient of $F_m(z)$, which gives the stated binomial coefficient.

\subsection{Notation and conventions}
To make the calculation more readable, we first establish some notation. Let $p$ be an odd prime, and work throughout over $\mathbb{F}_p$. Set $P := p-1$. Thus, $P$ is even and $P \equiv -1 \pmod p$.

Throughout this section, let $s$ and $r$ be integers satisfying $(p-1)s < r \leq ps$. Equivalently, we may write $r = Ps + d$ for $1 \leq d \leq s$. Fix an integer $0 \leq j < s$. We define the variables $m := s-j$, $N := d-j-1$, and $A_0 := s+1-d$. It follows that $N = m-A_0$.

We adopt the following conventions. First, if $$H(z) = \sum_{q \geq 0} h_q z^q \in \mathbb{F}_p[[z]],$$ then $[z^q]H(z) := h_q$ for $q \geq 0$, and we set $[z^q]H(z) := 0$ for $q < 0$. We use the same notation for polynomials, viewed as formal power series. All infinite sums and geometric-series expansions below are understood as identities of formal power series.

Second, we adopt the standard generalized binomial convention: $\binom{M}{q} = 0$ if $q < 0$. For a nonnegative upper index $M$, we also use the ordinary convention that $\binom{M}{q} = 0$ if $q > M$. For negative upper indices, the binomial coefficient does not automatically vanish.

\subsection{Generating functions for the coefficient families}
Both $S_j$ and $G_j$ are built from the coefficient families $a_k$ and $c_a(q)$. We first encode these families by generating functions, since the rest of the argument is a sequence of coefficient extractions.

From Lemma \ref{preimageLem}, we know that the coefficients take the form $a_k = (-1)^k\binom{-sP-1}{k}$ for $k \geq 0$. Their generating function is
$$A_s(z) := \sum_{k\geq 0} a_k z^k = (1-z)^{-sP-1}.$$
Moreover, from Proposition \ref{coefficientDecompositionProp}, for every integer $a \geq 0$, we have $c_a(q) = (-1)^q\binom{aP}{q}$ for $q \in \mathbb{Z}$, with the convention that $c_a(q) = 0$ for $q < 0$. The generating function of $c_a$ is
$$C_a(z) := \sum_{q\geq 0} c_a(q) z^q = (1-z)^{aP}.$$
Multiplying these two generating functions with $a=j$, we get
$$
A_s(z)C_j(z)=(1-z)^{-sP-1}(1-z)^{jP}=(1-z)^{-(s-j)P-1}.
$$

This motivates the following notation.

\begin{Def}[Combined coefficient series]
For each integer $M\geq 0$, define
$$
F_M(z):=(1-z)^{-MP-1}\in \mathbb F_p[[z]].
$$
Equivalently,
$$
F_M(z)
=
\sum_{q\geq 0}\binom{MP+q}{q}z^q
=
\sum_{q\geq 0}(-1)^q\binom{-MP-1}{q}z^q.
$$
\end{Def}

With the special value $M=m=s-j$, the product above becomes
$$
A_s(z)C_j(z)=F_m(z).
$$
Therefore, for every $q\in\mathbb Z$,
$$
\sum_{k\geq 0}a_kc_j(q-k)
=
[z^q]F_m(z),
$$
where the right-hand side is interpreted as zero if $q<0$.

By Proposition \ref{coefficientDecompositionProp},
the shear contribution is
$$
S_j=-\sum_{\substack{k,h\geq 0\\ k+h\leq m}}
a_k\binom{(r+k+h)P}{hP}c_j(m-k-h)
$$
and the swap contribution is
$$
G_j=\sum_{k\geq 0}a_kc_j(r+k-j). 
$$

Note that we have extended the $k$-sums beyond the original range $0 \leq k \leq s-1$. This does not change the value modulo $p$. For the shear term, if $j > 0$, then $m = s-j < s$, so no term with $k \geq s$ can satisfy $k+h \leq m$. If $j = 0$, the only additional possible term is $k = s$ and $h = 0$, and by Lucas' theorem, its coefficient is $a_s = (-1)^s\binom{-sP-1}{s} = \binom{sp}{s} \equiv 0 \pmod p$. For the swap term, terms with $k \geq s$ vanish because $r+k-j > jP$, which forces $c_j(r+k-j) = 0$.

\subsection{Computation of the shear term}
We first treat the shear term $S_j$. The main point is that after reindexing the sum by $n=k+h$, the inner sum can be recognized as the coefficient of a product involving a $P$-section polynomial.

We now introduce the $P$-section polynomial $D_A(z)$, which records only the coefficients whose exponents are divisible by $P$.

\begin{Def}
    For $A\geq 1$, define the $P$-section polynomial
    $$
    D_A(z):=\sum_{u=0}^A\binom{AP}{uP}z^u. 
    $$
    For the generating-function proof below, we also put $D_0(z):=1$.
\end{Def}

Since $P=p-1$ is even, $[t^{uP}](1-t)^{AP}=(-1)^{uP}\binom{AP}{uP}=\binom{AP}{uP}$. Thus, 
$$D_A(z)=\sum_{u\geq 0}[t^{uP}](1-t)^{AP}z^u.$$

\begin{Prop}\label{shearcontri}
    With the notation above, the shear contribution is
    $$
    S_j = - [z^m]F_m(z)D_{A_0}(z).
    $$
\end{Prop}

\begin{proof}
Start from
$$
S_j= -
\sum_{\substack{k,h\geq 0\\ k+h\leq m}}
a_k\binom{(r+k+h)P}{hP}c_j(m-k-h).
$$
Set $n:=k+h$. Then $0\leq n\leq m$ and $h=n-k$. Hence,
$$S_j= - \sum_{n=0}^m c_j(m-n)H_n,$$
where $H_n:=\sum_{k=0}^n a_k\binom{(r+n)P}{(n-k)P}$.
We compute $H_n$. Since $\binom{(r+n)P}{(n-k)P}=[x^{(n-k)P}](1+x)^{(r+n)P}$, we have
$$
\begin{aligned}
H_n
&=\sum_{k=0}^n a_k[x^{(n-k)P}](1+x)^{(r+n)P}\\
&=[x^{nP}](1+x)^{(r+n)P}\sum_{k\geq 0}a_kx^{kP}.
\end{aligned}
$$
Since $A_s(z) =\sum_{k \geq 0} a_k z^k = (1-z)^{-sP-1}$, this gives
$$H_n=[x^{nP}](1+x)^{(r+n)P}(1-x^P)^{-sP-1}.$$
Now use the standard coefficient-change formula for the substitution $x=t/(1-t)$:
$$[x^M]\Phi(x)=[t^M]\Phi\left(\frac{t}{1-t}\right)(1-t)^{M-1}. $$
Apply this with $M=nP$ and $\Phi(x)=(1+x)^{(r+n)P}(1-x^P)^{-sP-1}$. Since $1+x=\frac{1}{1-t}$, and, in $\mathbb{F}_p$,
$$
\begin{aligned}
1-x^P
&=1-\frac{t^P}{(1-t)^P}\\
&=\frac{(1-t)^P-t^P}{(1-t)^P}\\
&=\frac{1+t+\cdots+t^{P-1}}{(1-t)^P}\\
&=\frac{1-t^P}{(1-t)^{P+1}},
\end{aligned}
$$
we get
$$
\begin{aligned}
H_n
&=[t^{nP}](1-t)^{-(r+n)P}
\left(\frac{1-t^P}{(1-t)^{P+1}}\right)^{-sP-1}(1-t)^{nP-1}\\
&=[t^{nP}](1-t)^{-(r+n)P+nP-1+(P+1)(sP+1)}(1-t^P)^{-sP-1}.
\end{aligned}
$$
The exponent of $(1-t)$ simplifies as follows:
$$
\begin{aligned}
&-(r+n)P+nP-1+(P+1)(sP+1)\\
&=-rP-1+p(sP+1)\\
&=-(Ps+d)P-1+p(sP+1)\\
&=-sP^2-dP-1+psP+p\\
&=sP-dP+P\\
&=(s+1-d)P=A_0P.
\end{aligned}
$$
Therefore
$$
H_n=[t^{nP}](1-t)^{A_0P}(1-t^P)^{-sP-1}. 
$$
Since $D_{A_0}(z) = \sum_{u \geq 0}[t^{uP}] (1-t)^{A_0P}z^u$ and $A_s(z) = (1-z)^{-sP-1}$, this implies
$$
H (z) := \sum_{n\geq 0}H_nz^n=D_{A_0}(z)A_s(z).
$$
Finally, 
$$
\begin{aligned}
S_j
&= - \sum_{n=0}^m c_j(m-n)H_n\\
&=- [z^m]C_j(z)\sum_{n\geq 0}H_nz^n\\
&=- [z^m]C_j(z)D_{A_0}(z)A_s(z)\\
&=- [z^m]F_m(z)D_{A_0}(z),
\end{aligned}
$$
where the second equality holds because $$C_j(z) H(z) = (\sum_{q \geq 0} c_j(q) z^{q})(\sum_{n\geq 0}H_nz^n) = \sum_{q, n \geq 0} c_j(q) H_n z^{q+n},$$
and, taking the coefficient of $z^m$ and setting $q = m-n$, we have 
$$[z^m]C_j(z)H (z) = \sum_{q+n = m} c_j(q) H_n = \sum_{n=0}^{m} c_j(m-n) H_n.$$
\end{proof}

\subsection{Computation of the swap term}
We next treat the swap term $G_j$. The same polynomial $D_{A_0}(z)$ will appear again, but now through the difference $D_{A_0}(z)-z^{A_0}$. 

We first express $D_A(z)$ by an averaging formula over $\F_p^\times$.

\begin{Lem}\label{averagingDA} 
    Introduce a formal symbol $y$ with $y^P=z$. Then
    $$D_A(z) = \frac{1}{P} \sum_{\eta \in \F_p^{\times}} (1 + \eta y)^{AP}.$$
\end{Lem}

\begin{proof}
    Averaging over $\eta \in \mathbb F_p^\times$, we obtain
    $$\frac{1}{P} \sum_{\eta \in \F_p^{\times}} (1 + \eta y)^{AP} = \sum_{l = 0}^{AP} \binom{AP}{l} \Big(\frac{1}{P} \sum_{\eta \in \F_p^{\times}} \eta^l
    \Big) y^l.$$
    Now, $$
    \frac{1}{P} \sum_{\eta \in \F_p^{\times}} \eta^l
    =
    \begin{cases}
    1, & P\mid l,\\
    0, & P\nmid l.
    \end{cases}
    $$
    Therefore only terms with $l=uP$ survive. Hence
    $$\frac{1}{P} \sum_{\eta \in \F_p^{\times}} (1 + \eta y)^{AP} = \sum_{u = 0}^{A} \binom{AP}{uP} y^{uP} = \sum_{u=0}^A\binom{AP}{uP}z^u=D_A(z).$$
\end{proof}

We shall use the following identity later.

\begin{Lem}\label{HermiteId}
    The following identity holds in $\mathbb{F}_p[T,z]$:
$$
(1-z)\bigl((1-T)^P-z(1-z)^PT^p\bigr)=(1-Tz)\bigl((1-T)^P-z(1-Tz)^P\bigr). 
$$
\end{Lem}

\begin{proof}
    Over $\F_p$, we have the standard identity $X^p-Y^p=(X-Y)^p$.
    
Let $R$ denote the right-hand side and $L$ the left-hand side of the desired identity. Then
$$
\begin{aligned}
R-L
&=
\bigl((1-Tz)-(1-z)\bigr)(1-T)^P
-z(1-Tz)^{P+1}
+z(1-z)^{P+1}T^p \\
&=
z(1-T)^p
-z\bigl((1-Tz)^p-T^p(1-z)^p\bigr).
\end{aligned}
$$
Since in $\F_p$ we have
$$
(1-Tz)^p-T^p(1-z)^p
=
\bigl((1-Tz)-T(1-z)\bigr)^p
=
(1-T)^p,
$$
we get $R-L=0$. Therefore the identity holds.
\end{proof}

The next lemma is the key structural identity for $D_A(z)$. It shows that $D_A(z)-z^A$ is a polynomial in the single expression $z(1-z)^P$. This is the form that will allow the swap contribution to be compared directly with the shear contribution.

\begin{Lem}\label{DA-ZLem}
    For every $A\geq 1$, in $\mathbb{F}_p[z]$,
$$
\begin{aligned}
D_A(z)-z^A
& = \sum_{\omega\geq 0}[T^{A-p\omega-1}](1-T)^{-P\omega-1}\bigl(z(1-z)^P\bigr)^\omega \\
& = \sum_{0\leq \omega\leq \left\lfloor (A-1)/p\right\rfloor}
\binom{A-1-\omega}{P\omega}\bigl(z(1-z)^P\bigr)^\omega. 
\end{aligned}
$$
\end{Lem}

\begin{proof}
    We define 
    $$
    \mathcal{D}(T,z):=\sum_{B\geq 0}D_B(z)T^B,
    $$
    where $D_0(z) = 1$. By Lemma \ref{averagingDA}, we have
    $$\mathcal{D}(T, z) = \frac{1}{P} \sum_{\eta \in \F_p^{\times}} \sum_{B \geq 0} \big(T(1 + \eta y)^{P} \big)^B = \frac{1}{P} \sum_{\eta \in \F_p^{\times}} \frac{1}{1 - T(1+\eta y)^{P}}$$
    by the geometric series:
    $$\sum_{B \geq 0} \big(T(1 + \eta y)^P \big)^B = \frac{1}{1 - T(1+\eta y)^{P}}.$$
    Since $P = p-1$ and over $\F_p$, we have
    $$(1 + \eta y)^P = \frac{(1 + \eta y)^p}{1 + \eta y} = \frac{1 + \eta^p y^p}{1 + \eta y} = \frac{1 + \eta y^P y}{1 + \eta y} = \frac{1 + \eta z y}{1 + \eta y},$$
    where we set $y^P = z$ and used Fermat's little theorem for $\eta$.
    Now, by defining the variables $a := 1 - T$ and $b := y(1-Tz)$, we can simplify the expression:
    $$\frac{1}{1 - T(1+\eta y)^{P}} = \frac{1}{1 - T \frac{1 + \eta z y}{1 + \eta y}} = \frac{1 + \eta y}{a + \eta b}  =\frac{1}{1-Tz}+\frac{T(1-z)}{1-Tz}\cdot\frac{1}{a+\eta b},$$
    where the second equality follows from multiplying numerator and denominator by $1+\eta y$ and the last equality follows from a standard partial fraction split.
    Averaging the second factor gives
    $$
    \begin{aligned}
    \frac{1}{P}\sum_{\eta\in\mathbb{F}_p^\times}\frac{1}{a+\eta b}
    &=\frac{1}{P}\sum_{\eta\in\mathbb{F}_p^\times}\frac{1}{a}\sum_{n\geq 0}(-1)^n\eta^n\left(\frac{b}{a}\right)^n\\
    &=\frac{1}{a}\sum_{q\geq 0}\left(\frac{b}{a}\right)^{qP}\\
    &=\frac{a^{P-1}}{a^P-b^P},
    \end{aligned}
    $$
    where we have again used the identity: 
    $$
    \frac{1}{P} \sum_{\eta \in \F_p^{\times}} \eta^n
    =
    \begin{cases}
    1, & P\mid n,\\
    0, & P\nmid n.
    \end{cases}
    $$
    and, in the surviving terms, $n=qP$; the sign $(-1)^n$ is then $1$ because $P$ is even.

    Since $a=1-T$ and $b^P=z(1-Tz)^P$, this yields
    $$
    \mathcal{D}(T,z)=\frac{1}{1-Tz}+\frac{T(1-z)(1-T)^{P-1}}{(1-Tz)\bigl((1-T)^P-z(1-Tz)^P\bigr)}. 
    $$
    By Lemma \ref{HermiteId}, we can rewrite $\mathcal{D}(T, z)$ as
    $$\mathcal{D}(T,z)=\frac{1}{1-Tz}+\frac{T(1-T)^{P-1}}{(1-T)^P-z(1-z)^PT^p}.$$ 

    Since $\frac{1}{1-Tz} = \sum_{B\geq 0}z^BT^B$ and $\mathcal{D}(T,z) = \sum_{B \geq 0} D_B(z) T^B$, we obtain
    $$
    \begin{aligned}
    \sum_{B\geq 0}(D_B(z)-z^B)T^B
    &=\frac{T(1-T)^{P-1}}{(1-T)^P-z(1-z)^PT^p}\\
    &=\frac{T}{1-T}\cdot \frac{1}{1-\dfrac{z(1-z)^PT^p}{(1-T)^P}}\\
    &=\sum_{\omega\geq 0}\bigl(z(1-z)^P\bigr)^\omega T^{p\omega+1}(1-T)^{-P\omega-1}.
    \end{aligned}
    $$
    Taking the coefficient of $T^A$ gives the first equality of the statement. 

    Now, for the second equality of the statement, if $A-p\omega-1<0$, then
    $$
    {}[T^{A-p\omega-1}](1-T)^{-P\omega-1}=0
    $$
    because this is a negative power of $T$ in an ordinary power series. If $A-p\omega-1\geq 0$, which is equivalent to $0\leq \omega\leq \left\lfloor \frac{A-1}{p}\right\rfloor$, then
    $$
    \begin{aligned}
    {}[T^{A-p\omega-1}](1-T)^{-P\omega-1}
    &=\binom{P\omega+(A-p\omega-1)}{A-p\omega-1}\\
    &=\binom{A-1-\omega}{A-p\omega-1}\\
    &=\binom{A-1-\omega}{P\omega}.
    \end{aligned}
    $$
\end{proof}

The following lemma explains why the previous identity is useful. Multiplication by $F_M(z)$ and extraction of $[z^M]$ detects the $M$-th power of $W=z(1-z)^P$.

\begin{Lem}\label{detectorlem}
For every $A\geq 1$ and every $M\geq 0$,
$$
[z^M]F_M(z)(D_A(z)-z^A)
=
[z^{A-1-pM}]F_M(z).
$$
The coefficient on the right is interpreted as zero if $A-1-pM<0$.
\end{Lem}

\begin{proof}
Let $W := z(1-z)^P$. We first record the detector identity:
$$
[z^M]F_M(z)W^\omega
=
\begin{cases}
1, & \omega=M,\\
0, & \omega\neq M.
\end{cases}
$$

Indeed, if $\omega>M$, then after removing the factor $z^\omega$, we would need the coefficient of a negative power of $z$, so the coefficient is zero. If $\omega=M$, then
$$
[z^M]F_M(z)W^M
=
[z^0](1-z)^{-1}
=
1.
$$
Finally, if $\omega<M$, set $n:=M-\omega>0$. Then
$$
[z^M]F_M(z)W^\omega
=
[z^n](1-z)^{-nP-1}
=
\binom{nP+n}{n}
=
\binom{np}{n}
\equiv 0 \pmod p
$$
by Lucas' theorem.

Now Lemma \ref{DA-ZLem} gives
$$
D_A(z)-z^A
=
\sum_{0\leq \omega\leq \left\lfloor (A-1)/p\right\rfloor}
\binom{A-1-\omega}{P\omega}W^\omega.
$$
Applying the detector identity, only the term $\omega=M$ can contribute. Therefore
$$
[z^M]F_M(z)(D_A(z)-z^A)
=
\begin{cases}
\binom{A-1-M}{PM}, & A-1-pM\geq 0,\\
0, & A-1-pM<0.
\end{cases}
$$

It remains to identify this with the coefficient on the right. If $A-1-pM<0$, then by convention
$$
[z^{A-1-pM}]F_M(z)=0.
$$
If $A-1-pM\geq 0$, then
$$
\begin{aligned}
{}[z^{A-1-pM}]F_M(z)
&=
[z^{A-1-pM}](1-z)^{-MP-1}\\
&=
\binom{MP+A-1-pM}{A-1-pM}\\
&=
\binom{A-1-M}{PM}.
\end{aligned}
$$
Thus the two sides agree in all cases.
\end{proof}

We can now return to the swap contribution. Vandermonde's identity first rewrites $G_j$ as a coefficient of $F_m(z)$, then the previous lemma converts this coefficient into an expression involving $D_{A_0}(z)-z^{A_0}$.

\begin{Prop}\label{swapcontri}
The swap contribution $G_j$ satisfies
$$
G_j=[z^m]F_m(z)(D_{A_0}(z)-z^{A_0}).
$$
\end{Prop}

\begin{proof}
    Starting from $G_j=\sum_{k\geq 0}a_kc_j(r+k-j)$ and using the definitions,
$$
G_j=(-1)^{r-j}\sum_{k\geq 0}\binom{-sP-1}{k}\binom{jP}{r+k-j}.
$$
Let $\Lambda := jP+j-r = jp-r$. Then $jP-(r+k-j) = \Lambda-k$. Therefore, by Vandermonde's identity,
$$
\begin{aligned}
G_j
&=(-1)^{r-j}\sum_{k\geq 0}\binom{-sP-1}{k}\binom{jP}{\Lambda-k}\\
&=(-1)^{r-j}\binom{-sP-1+jP}{\Lambda}\\
&=(-1)^{r-j}\binom{-mP-1}{\Lambda}.
\end{aligned}
$$
If $\Lambda < 0$, this term vanishes. For $\Lambda \geq 0$, we have $\binom{-mP-1}{\Lambda} = (-1)^\Lambda \binom{mP+\Lambda}{\Lambda}$. Thus, $G_j = (-1)^{r-j+\Lambda} \binom{mP+\Lambda}{\Lambda}$. Because $r-j+\Lambda = r-j+jp-r = jP$, which is an even integer, it follows that for $\Lambda \geq 0$, $G_j = \binom{mP+\Lambda}{\Lambda} = [z^\Lambda]F_m(z)$.

Combining this with the $\Lambda < 0$ case, and applying the convention that $[z^\Lambda]F_m(z) = 0$ for $\Lambda < 0$, we conclude that $G_j = [z^\Lambda]F_m(z)$ for all integers $\Lambda$.

Finally,
$$
\begin{aligned}
\Lambda
&=jp-r\\
&=jp-(Ps+d)\\
&=s-d-p(s-j)\\
&=A_0-1-pm.
\end{aligned}
$$
By Lemma \ref{detectorlem} with $A = A_0$ and $M=m$,
$$
[z^{A_0-1-pm}]F_m(z)
=
[z^m]F_m(z)(D_{A_0}(z)-z^{A_0}).
$$
Since $\Lambda=A_0-1-pm$, this gives
$$
G_j=[z^\Lambda]F_m(z)
=
[z^m]F_m(z)(D_{A_0}(z)-z^{A_0}).
$$

\end{proof}

\subsection{Cancellation of the shear and swap contributions}
We now combine the two coefficient expressions. The terms involving $D_{A_0}(z)$ cancel, leaving only the shifted coefficient of $F_m(z)$. Together, this proves Proposition \ref{interiorCoefficientProp}.

\begin{proof}[Proof of Proposition \ref{interiorCoefficientProp}]
From Propositions \ref{shearcontri} and \ref{swapcontri}, we have $S_j = -[z^m]F_m(z)D_{A_0}(z)$ and $G_j = [z^m]F_m(z)(D_{A_0}(z)-z^{A_0})$.

Summing these contributions yields:
$$
\begin{aligned}
S_j+G_j
&=-[z^m]F_m(z)D_{A_0}(z)+[z^m]F_m(z)(D_{A_0}(z)-z^{A_0})\\
&=-[z^m]F_m(z)z^{A_0}\\
&=-[z^{m-A_0}]F_m(z).
\end{aligned}
$$
Since $m-A_0=N$, this gives $S_j+G_j=-[z^N]F_m(z)$. Substituting the definitions $F_m(z)=(1-z)^{-mP-1}$, $m = s-j$, and $N = d-j-1$, this becomes $S_j+G_j=-[z^{d-j-1}](1-z)^{-(s-j)P-1}$. Expressing this in terms of the original variables $r$, $s$, and $p$ via the generalized binomial formula, we obtain
$$S_j+G_j = (-1)^{r-j} \binom{(j-s)(p-1)-1}{r-(p-1)s-j-1}.$$
Finally, for $0 \leq j < s$, the total coefficient is given by $C_j = I_j + S_j + G_j$. Since $I_j = 0$ in this range, the sum $S_j + G_j = C_j$, completing the proof.
\end{proof}

\clearpage
\printbibliography

\end{document}